\documentclass[11pt,a4paper]{article}

\pdfoutput=1 % for arxiv
\usepackage{epsf,epsfig,amsfonts,amsgen,amsmath,amstext,amsbsy,amsopn,amsthm,cases,listings,color
}
\usepackage{ebezier,eepic}
\usepackage{color}
\usepackage{multirow}
\usepackage{epstopdf}
\usepackage{graphicx}
\usepackage{pgf,tikz}
\usepackage{mathrsfs}
\usepackage[marginal]{footmisc}
\usepackage{enumitem}
\usepackage[titletoc]{appendix}
\usepackage{booktabs}
\usepackage{url}
\usepackage{mathtools}

\usepackage{pgfplots}
\usepackage{authblk}
\usepackage{amssymb}

\usepackage{wasysym}

\usepackage{empheq}

\usepackage{dsfont}
\usepackage{verbatim}

\usepackage{caption}
\usepackage{subcaption}
\pgfplotsset{compat=1.18}
\usepackage{mathrsfs}
\usepackage{wasysym} 
\usetikzlibrary{arrows}
\usepackage{aligned-overset}
\usepackage{bm}
\usepackage[backref=page]{hyperref}
\usepackage{makecell}

\newcommand{\Mod}[1]{\ (\mathrm{mod}\ #1)}

\allowdisplaybreaks[1]

\definecolor{uuuuuu}{rgb}{0.27,0.27,0.27}
\definecolor{sqsqsq}{rgb}{0.1255,0.1255,0.1255}

\newtheorem{definition}{Definition} [section]
\newtheorem{theorem}[definition]{Theorem}
\newtheorem{lemma}[definition]{Lemma}
\newtheorem{proposition}[definition]{Proposition}

\newtheorem{corollary}[definition]{Corollary}

\newtheorem{claim}[definition]{Claim}

\newtheorem{fact}[definition]{Fact}
\newenvironment{pf}{\noindent{\bf Proof}}

\begin{document}
%%%%%%%%%%%%%%%%%%%%%%%%%%%%%%%%%%%%%%%%%%%%%%%%%%%%%%%
\title{\bf\Large Positive codegree Andr{\'a}sfai--Erd\H{o}s--S\'{o}s theorem for the generalized triangle}
\date{\today}
%%%%%%%%%%%%%%%%%%%%%%%%%%%%%%%%%%%%%%%%%%%%%%%%%%%%
\author[1]{Xizhi Liu\thanks{Research supported by ERC Advanced Grant 101020255. Email: \texttt{xizhi.liu.ac@gmail.com}}}
\author[2]{Sijie Ren\thanks{Email: \texttt{rensijie1@126.com}}}
\author[2]{Jian Wang\thanks{Research supported by National Natural Science Foundation of China No. 12471316. Email: \texttt{wangjian01@tyut.edu.cn}}}
% %%%%%%%%%%%%%%%%%%%%%%%%%%%%%%%%%%%%%%%%%%%%%%%%%%%%%
\affil[1]{Mathematics Institute and DIMAP,
            University of Warwick,
            Coventry, CV4 7AL, UK}
\affil[2]{Department of Mathematics, 
            Taiyuan University of Technology, Taiyuan, 030024, China}
%%%%%%%%%%%%%%%%%%%%%%%%%%%%%%%%%%%%%%%%%%%%%%%%%%%
\maketitle
%\footnote{footnote}
%%%%%%%%%%%%%%%%%%%%%%%%%%%%%%%%%%%%%%%%%%%%%%%%%
%%%%%%%%%%%%%%%%%%%%%%%%%%%
\begin{abstract}
The celebrated Andr{\'a}sfai--Erd\H{o}s--S\'{o}s Theorem from 1974 shows that every $n$-vertex triangle-free graph with minimum degree greater than $2n/5$ must be bipartite.

We establish a positive codegree extension of this result for the $r$-uniform generalized triangle $\mathrm{T}_{r} = \left\{\{1,\ldots, r-1,r\}, \{1,\ldots, r-1,r+1\},\{r,r+1, \ldots, 2r-1\}\right\}$$\colon$ For every $n \ge (r-1)(2r+1)/2$, if $\mathcal{H}$ is an $n$-vertex $\mathrm{T}_{r}$-free $r$-uniform hypergraph in which each $(r-1)$-tuple of vertices is contained in either zero edges or more than $2n/(2r+1)$ edges of $\mathcal{H}$, then $\mathcal{H}$ is $r$-partite. 
This result provides the first tight positive codegree Andr{\'a}sfai--Erd\H{o}s--S\'{o}s type theorem for hypergraphs.
It also immediately implies that the positive codegree Tur\'{a}n number of $\mathrm{T}_{r}$ is $\lfloor n/r \rfloor$ for all $r$. 
Additionally, for $r=3$, our result answers one of the questions posed by Hou et al.~\cite{HLYZZ22} in a strong form.

\medskip

\noindent\textbf{Keywords:} Andr{\'a}sfai--Erd\H{o}s--S\'{o}s theorem, generalized triangle, cancellative hypergraph, positive codegree
% \medskip
% \textbf{MSC2020:} 	05C65, 05C35, 05D99. 
%Find suitable code from https://mathscinet.ams.org/msc/msc2010.html
\end{abstract}
%%%%%%%%%%%%%%%%%%%%%%%%%%%%%%%%%%%%%%%%%%%
\section{Introduction}\label{SEC:Intorduction}
Given an integer $r\ge 2$, an \textbf{$r$-uniform hypergraph} (henceforth \textbf{$r$-graph}) $\mathcal{H}$ is a collection of $r$-subsets of some finite set $V$.
We identify a hypergraph $\mathcal{H}$ with its edge set and use $V(\mathcal{H})$ to denote its vertex set. 
The size of $V(\mathcal{H})$ is denoted by $v(\mathcal{H})$. 
The \textbf{degree} $d_{\mathcal{H}}(v)$ of a vertex $v$ in $\mathcal{H}$ is the number of edges containing $v$. 
We use $\delta(\mathcal{H})$, $\Delta(\mathcal{H})$, and $d(\mathcal{H})$ to denote the \textbf{minimum}, \textbf{maximum}, and \textbf{average degree} of $\mathcal{H}$, respectively.
The \textbf{shadow} $\partial\mathcal{H}$ of $\mathcal{H}$ is defined as 
\begin{align*}
    \partial\mathcal{H}
    \coloneqq 
    \left\{e\in \binom{V(\mathcal{H})}{r-1} \colon \text{there exists $E\in \mathcal{H}$ with $e\subseteq E$}\right\}.
\end{align*}
For every $e\in \partial\mathcal{H}$, the \textbf{neighborhood} $N_{\mathcal{H}}(e)$ of $e$ in $\mathcal{H}$ is defined as 
\begin{align*}
    N_{\mathcal{H}}(e)
    \coloneqq \left\{v\in V(\mathcal{H}) \colon e\cup \{v\} \in \mathcal{H}\right\}. 
\end{align*}
Following the definition of Balogh--Lemons--Palmer~\cite{BLP21}, the \textbf{minimum positive codegree}\footnote{This is also referred to as the \textbf{minimum shadow degree} in recent work~\cite{FW24}.} of $\mathcal{H}$ is defined as
\begin{align*}
    \delta_{r-1}^{+}(\mathcal{H})
    \coloneqq 
    \min\left\{|N_{\mathcal{H}}(e)| \colon e \in \partial \mathcal{H}\right\}.
\end{align*}
Note that in the case $r=2$, the minimum positive codegree of a graph without isolated vertices is equal to its minimum degree.

In a seminal work~\cite{AES74}, Andr\'{a}sfai--Erd\H{o}s--S\'{o}s showed that for $\ell \ge 2$, every $K_{\ell+1}$-free graph $G$ on $n$ vertices with minimum degree greater than $\frac{3\ell-4}{3\ell-1} n$ must be $\ell$-partite. Moreover, the bound $\frac{3\ell-4}{3\ell-1} n$ is tight. Extensions of this result to hypergraphs have been explored in several prior works such as~\cite{FS05,KS05,FPS06Book,LMR23unif,HLZ24,CL24}, though the first tight result for hypergraphs was established only recently in~\cite{LRW24} for the $3$-uniform generalized triangle $\mathrm{T}_{3}$.
Here, for $r \ge 2$, the $r$-uniform generalized triangle $\mathrm{T}_{r}$ is defined as
\begin{align*}
    \mathrm{T}_{r} 
    \coloneqq \left\{\{1,\ldots, r-1,r\}, \{1,\ldots, r-1,r+1\},\{r,r+1, \ldots, 2r-1\}\right\},
\end{align*}
which is a classic example that has been extensively studied by numerous researchers since the 1960s (see e.g.~\cite{Bol74,FF83F5,Sid87,FF89,She96,KM04Cancel,Goldwasser,Pik08,BM12,KLM14,BBHLM16,NY18,Liu21Cancel,LM21feasbile,HLP22,BCL22b,LM23KK,Liu24Steiner,CILLP24}).
\begin{theorem}[{\cite[Theorem~1.1]{LRW24}}]\label{THM:AES-F5-degree}
    For $n \ge 5000$, every $n$-vertex $F_5$-free $3$-graph with $\delta(\mathcal{H}) > \frac{4n^2}{45}$ is $3$-partite.
\end{theorem}
In this note, we consider the minimum positive codegree extension of the Andr\'{a}sfai--Erd\H{o}s--S\'{o}s Theorem for $\mathrm{T}_{r}$$\colon$ 
What is the smallest real number $\alpha_{r}$ such that every $n$-vertex $\mathrm{T}_{r}$-free $r$-graph with $\delta_{r-1}^{+}(\mathcal{H}) > \alpha_{r} n$ is $r$-partite?
The Andr\'{a}sfai--Erd\H{o}s--S\'{o}s Theorem shows that $\alpha_2 = \frac{2}{5}$. 
A classical result by Frankl--F\"{u}redi~\cite{FF83F5} implies that $\alpha_{3} \le \frac{1}{3}$. Similar upper bound for general $\alpha_r$ can be obtained from results in~\cite{Sid87,Pik08,FF89}. 
In the following theorem, we determine the exact value of $\alpha_{r}$ for every $r\ge 3$. 
\begin{theorem}\label{THM:main-AES-Tr-codegree}
    For $n \ge \frac{(r-1)(2r+1)}{2}$, every $n$-vertex $\mathrm{T}_{r}$-free $r$-graph $\mathcal{H}$ with $\delta_{r-1}^{+}(\mathcal{H}) > \frac{2n}{2r+1}$ is $r$-partite. 
\end{theorem}
\textbf{Remarks.}
\begin{itemize}
    \item Theorem~\ref{THM:main-AES-Tr-codegree} does not hold for $r+1 \le n < \frac{(r-1)(2r+1)}{2}$. For example, when $2r-2 \le n < \frac{(r-1)(2r+1)}{2}$, the $r$-graph $\mathcal{H}$ consisting of the complete $r$-graph $K_{2r-2}^{r}$ on $2r-2$ vertices, along with $n-(2r-2)$ isolated vertices, is $\mathrm{T}_{r}$-free and satisfies $\delta_{r-1}^{+}(\mathcal{H}) = \delta_{r-1}^{+}(K_{2r-2}^{r}) = r-1 > \frac{2\cdot n}{2r+1}$, yet it is not $r$-partite. Constructions for smaller $n$ can be obtained similarly.   
    \item The bound $\frac{2n}{2r+1}$ in Theorem~\ref{THM:main-AES-Tr-codegree} is optimal, as shown by the following construction.
    %%%%%%%%%%%%%%%%%%%%%%%%%%%%%
    \begin{figure}[h]
    \centering
    %%%%%%%%%%%%%%%%%%%%%%%%%%%%%%%%
    \tikzset{every picture/.style={line width=1pt}} %set default line width to 0.75pt         
    
    \begin{tikzpicture}[x=0.75pt,y=0.75pt,yscale=-0.8,xscale=0.8]
    %uncomment if require: \path (0,250); %set diagram left start at 0, and has height of 250
    %
    \draw [fill=uuuuuu, fill opacity=0.2, join = round]  (328.8,35.06) .. controls (346.33,59.4) and (343.33,98.4) .. (328.26,122.76) .. controls (356.33,95.4) and (370.33,90.4) .. (411.83,96.17) -- (328.8,35.06); 
    \draw [fill=uuuuuu, fill opacity=0.2, join = round]  (328.26,122.76) .. controls (364.33,129.4) and (391.33,121.4) .. (411.83,96.17) -- (379.38,194.02) .. controls (380.33,156.4) and (367.33,138.4) .. (328.26,122.76); 
    \draw [fill=uuuuuu, fill opacity=0.2, join = round]   (328.26,122.76) .. controls (322.33,155.4) and (309.33,178.4) .. (276.28,193.39) -- (379.38,194.02) .. controls (345.33,172.4) and (335.33,156.4) .. (328.26,122.76); 
    \draw [fill=uuuuuu, fill opacity=0.2, join = round]  (328.26,122.76) .. controls (295.33,137.4) and (281.33,156.4) .. (276.28,193.39) -- (245.03,95.15) .. controls (275.33,125.4) and (289.33,127.4) .. (328.26,122.76);
    \draw [fill=uuuuuu, fill opacity=0.2, join = round]  (245.03,95.15) .. controls (288.33,87.4) and (304.33,92.4) .. (328.26,122.76) .. controls (313.33,89.4) and (310.33,69.4) .. (328.8,35.06) -- (245.03,95.15);
    \draw [fill=uuuuuu] (411.83,96.17) circle (2.5pt);
    \draw [fill=uuuuuu] (379.38,194.02) circle (2.5pt);
    \draw [fill=uuuuuu] (276.28,193.39) circle (2.5pt);
    \draw [fill=uuuuuu] (245.03,95.15) circle (2.5pt);
    \draw [fill=uuuuuu] (328.8,35.06) circle (2.5pt);
    \draw [fill=uuuuuu] (328.26,122.76) circle (2.5pt);
    \end{tikzpicture}
    %%%%%%%%%%%%%%%%%%%%%%%%%%%%
    \caption{The $3$-uniform $5$-wheel $W_{5}^{3}$.} 
    \label{Fig:W53}
    \end{figure}
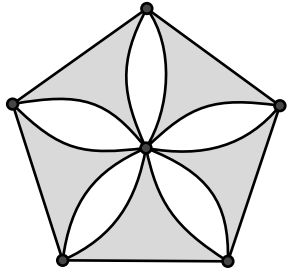
    
    Let the $r$-uniform $5$-wheel $W_{5}^{r}$ be the $r$-graph on $r+3$ vertices with edge set 
    \begin{align*}
        \left\{\{u_{1},\ldots u_{r-2},v_1,v_2\}, \{u_{1},\ldots u_{r-2},v_2,v_3\}, \ldots, \{u_{1},\cdots u_{r-2},v_5,v_1\} \right\}. 
    \end{align*}
    Given a tuple $(x_1, \ldots, x_{r-2}, y_1, \ldots, y_5) \in \mathbb{N}^{r+3}$, the blowup $W_{5}^{r}[x_1, \ldots, x_{r-2}, y_1, \ldots, y_5]$ of $W_{5}^{r}$ is obtained by replacing each $u_{i}$ with a set of size $x_i$, replacing each $v_i$ with a set of size $y_i$, and replacing each edge with the corresponding complete $r$-partite $r$-graph.  
    It is easy to see that every blowup of $W_{5}^{r}$ is $\mathrm{T}_{r}$-free and not $r$-partite. 
    % \begin{align*}
    %  \delta_{r-1}^+(W_{5}^{3}[x_1, \ldots, x_{r-2}, y_1, \ldots, y_5])
    %     = \min\left\{x,~y_{1}+y_{3},~ \ldots,~y_{5}+y_{2}\right\}, 
    % \end{align*}
    %
    % where the indices are taken modulo $5$. 
    
    Let $n$ be an integer satisfying $n \equiv 0 \Mod{2r+1}$, let $(x_1, \ldots, x_{r-2}, y_1, \ldots, y_5) \coloneqq \left(\frac{2n}{2r+1}, \ldots, \frac{2n}{2r+1}, \frac{n}{2r+1}, \ldots, \frac{n}{2r+1}\right)$, and let $\mathcal{G} \coloneqq W_{5}^{r}[x_1, \ldots, x_{r-2}, y_1, \ldots, y_5]$. 
    Simple calculations show that $\mathcal{G}$ has exactly $n$ vertices and $\delta_{r-1}^+(\mathcal{G})=\frac{2n}{2r+1}$. This shows that the bound $\delta_{r-1}^+(\mathcal{H})>\frac{2n}{2r+1}$ in Theorem~\ref{THM:main-AES-Tr-codegree}  cannot be improved in general. 
    \item A general theorem of Hou et al.~{\cite[Theorem~1.6]{HLYZZ22}} shows that if $\mathcal{H}$ is a \textbf{saturated}\footnote{Here, `saturated' means that $\partial \mathcal{H}$ is $K_4$-free, but adding any new edge into $\mathcal{H}$ will create a copy of $K_4$ in $\partial \mathcal{H}$.} $n$-vertex $3$-graph satisfying $K_4\not\subseteq \partial \mathcal{H}$ and $\delta_{2}^{+}(\mathcal{H}) > 2n/7$, then $\mathcal{H}$ is $3$-partite. They asked whether the constraint `saturated' can be removed. Observe that the shadow of $F_5$ contains a copy of $K_4$, and hence every $3$-graph $\mathcal{H}$ satisfying $K_4 \not\subseteq \partial \mathcal{H}$ must be $\mathrm{T}_{3}$-free. Thus, Theorem~\ref{THM:main-AES-Tr-codegree} for the case $r=3$ provides a positive answer to their question. 
\end{itemize}

A immediate consequence of Theorem~\ref{THM:main-AES-Tr-codegree} is the determination of the \textbf{positive codegree Tur\'{a}n number} $\mathrm{co}^{+}\mathrm{ex}(n,\mathrm{T}_{r})$ of $\mathrm{T}_{r}$, which is defined as 
\begin{align*}
    \mathrm{co}^{+}\mathrm{ex}(n,\mathrm{T}_{r})
    \coloneqq \max\left\{\delta^{+}_{r-1}(\mathcal{H}) \colon \text{$\mathcal{H}$ is an $n$-vertex $\mathrm{T}_{r}$-free $r$-graph}\right\}.
\end{align*}

\begin{corollary}\label{CORO:codegree-Turan-Tr}
    For $r\ge 3$ and $n \ge \frac{(r-1)(2r+1)}{2}$, we have 
    \begin{align*}
        \mathrm{co}^{+}\mathrm{ex}(n,\mathrm{T}_{r}) 
        = \left\lfloor \frac{n}{r} \right\rfloor.
    \end{align*}
\end{corollary}
\textbf{Remarks.} 
\begin{itemize}
    \item Corollary~\ref{CORO:codegree-Turan-Tr} extends~{\cite[Theorem~7]{HLP22}}, which establishes that $\mathrm{co}^{+}\mathrm{ex}(n,\mathrm{T}_{3}) = \lfloor n/3 \rfloor$ for $n \ge 6$. 
    \item Note that the extremal construction for the positive codegree Tur\'{a}n problem of $\mathrm{T}_{r}$ is the balanced $r$-partite $r$-graph on $n$-vertices for every $r\ge 2$. This construction coincides with the extremal construction for the Tur\'{a}n problem when $r \in \{2,3,4\}$. However, for $r \in \{5,6\}$, Corollary~\ref{CORO:codegree-Turan-Tr}, together with a classical theorem by Frankl--F\"{u}redi~\cite{FF89}, shows that these two problems have quite different extremal constructions. For $r \ge 7$, the  Tur\'{a}n problem for $\mathrm{T}_{r}$ is still wide open. 
\end{itemize}

%%%%%%%%%%%%%%%%%%%%%%%%%%%%%%%%%%%%%%%%%
\section{Proof of Theorem~\ref{THM:main-AES-Tr-codegree}}\label{SEC:proof-main}
We prove Theorem~\ref{THM:main-AES-Tr-codegree} in this section. Let us start with some necessary definitions and preliminary results.

Let $\mathcal{H}$ be an $r$-graph.
For a vertex $v \in V(\mathcal{H})$, the \textbf{link} of $v$ in $\mathcal{H}$ is defined as 
\begin{align*}
    L_{\mathcal{H}}(v)
    \coloneqq \left\{e\in \partial\mathcal{H} \colon e\cup \{v\} \in \mathcal{H}\right\}. 
\end{align*}
The \textbf{neighborhood} of $v$ in $\mathcal{H}$ is defined as 
\begin{align*}
    N_{\mathcal{H}}(v)
    \coloneqq \left\{u \in V(\mathcal{H})\setminus \{v\} \colon \text{there exists $e\in \mathcal{H}$ such that $\{u,v\} \subseteq e$} \right\}.
\end{align*}
Recall that for every $e\in \partial\mathcal{H}$, the \textbf{neighborhood} $N_{\mathcal{H}}(e)$ of $e$ in $\mathcal{H}$ is 
\begin{align*}
    N_{\mathcal{H}}(e)
    \coloneqq \left\{v\in V(\mathcal{H}) \colon e\cup \{v\} \in \mathcal{H}\right\}. 
\end{align*}
It is clear from the definition that for every $v \in V(\mathcal{H})$,  
\begin{align}\label{equ:coneighbor-neighbor}
    \bigcup_{e\in \partial\mathcal{H} \colon v\in e} N_{\mathcal{H}}(e) 
    = N_{\mathcal{H}}(v). 
\end{align}
We say a set $I \subseteq V(\mathcal{H})$ of vertices is \textbf{independent} in $\mathcal{H}$ if every edge in $\mathcal{H}$ contains at most one vertex in $I$. 

For every $r\ge 2$. Let $\Sigma_{r}$ denote the family of $r$-graphs consisting of three edges $A,B,C$ such that $|A\cap B| = r-1$ and the symmetric difference $A \triangle B$ is contained in $C$. Note that $\mathrm{T}_{r} \in \Sigma_{r}$. 

\begin{proposition}\label{PROP:Tr-free-implies-Sigma-r-free}
    Let $r\ge 3$ be an integer. Every $\mathrm{T}_{r}$-free $r$-graph $\mathcal{H}$ with $\delta_{r-1}^+(\mathcal{H}) \ge  r$ is also $\Sigma_{r}$-free. 
\end{proposition}
\begin{proof}[Proof of Proposition~\ref{PROP:Tr-free-implies-Sigma-r-free}]
    Let $\mathcal{H}$ be a $\mathrm{T}_{r}$-free $r$-graph satisfying $\delta_{r-1}^+(\mathcal{H}) \ge  r$. Suppose to the contrary that there exist three edges $A,B,C \in \mathcal{H}$ such that $|A\cap B| = r-1$ and the symmetric difference $A \triangle B$ is contained in $C$. 
    Assume that $A = \{u_1, \ldots, u_{r-1}, u_{r}\}$ and $B = \{u_1, \ldots, u_{r-1}, u_{r+1}\}$. We can choose such an edge $C \in \mathcal{H}$ so that $|C \cap \{u_1, \ldots, u_{r-1}\}|$ is minimized. Note that we obtain a contradiction if $|C \cap \{u_1, \ldots, u_{r-1}\}| = 0$, since
    \begin{align*}
        \{\{u_1, \ldots, u_{r-1},u_{r}\}, \{u_1, \ldots, u_{r-1},u_{r+1}\}, C\}
    \end{align*}
    would form a copy of $\mathrm{T}_{r}$ in $\mathcal{H}$. So we may assume that $|C \cap \{u_1, \ldots, u_{r-1}\}| \ge 1$. 

    By symmetry, we may assume that $C \cap \{u_1, \ldots, u_{r-1}\} = \{u_1, \ldots, u_{i_{\ast}}\}$ for some $i_{\ast} \in [r-2]$. Let $e \coloneqq C\setminus \{u_1\}$. Since $e\in \partial\mathcal{H}$ and $\delta_{r-1}^+(\mathcal{H}) \ge  r$, there exists a vertex $w \in V(\mathcal{H}) \setminus \{u_1, \ldots, u_{r-1}\}$ such that $C' \coloneqq e \cup \{w\} \in \mathcal{H}$. However, since $\{u_{r}, u_{r+1}\} \subseteq C'$ and $|C' \cap \{u_1, \ldots, u_{r-1}\}| < |C \cap \{u_1, \ldots, u_{r-1}\}|$, this contradicts the choice of $C$. 
    %This proves Proposition~\ref{PROP:Tr-free-implies-Sigma-r-free}. 
\end{proof}

\begin{lemma}\label{LEMMA:F5-3-indep-sets}
    Let $\mathcal{H}$ be a $\mathrm{T}_{r}$-free $r$-graph and $e = \{u_1,\ldots, u_r\} \in \mathcal{H}$ be an edge.
    Suppose that $\delta_{r-1}^+(\mathcal{H}) \ge  r$. Then 
    \begin{align}\label{equ:Nu-Nvw-empty}
         N_{\mathcal{H}}(e\setminus\{u_i\}) \cap N_{\mathcal{H}}(u_i)
        = \emptyset
        \quad\text{for every}\quad i\in [r]. 
    \end{align}
    Moreover, sets in $\left\{N_{\mathcal{H}}(e\setminus \{u_i\})\colon i\in [r] \right\}$ are pairwise disjoint and independent. 
\end{lemma} 
\begin{proof}[Proof of Lemma~\ref{LEMMA:F5-3-indep-sets}]
    By symmetry, it suffices to prove~\eqref{equ:Nu-Nvw-empty} for the case where $i=r$. Suppose to the contrary that there exists a vertex $u_{r+1} \in N_{\mathcal{H}}(e\setminus\{u_{r}\}) \cap N_{\mathcal{H}}(u_{r})$. By the definition of $N_{\mathcal{H}}(u_{r})$, there exists an edge $e\in \mathcal{H}$ containing $\{u_{r}, u_{r+1}\}$. Notice that $\{\{u_1, \ldots, u_{r-1},u_{r}\}, \{u_1, \ldots, u_{r-1},u_{r+1}\},e\}$ is a member in $\Sigma_{r}$, which, by Proposition~\ref{PROP:Tr-free-implies-Sigma-r-free}, is a contradiction. 
    
    It follows from~\eqref{equ:Nu-Nvw-empty} that sets in $\left\{N_{\mathcal{H}}(e\setminus \{u_i\})\colon i\in [r] \right\}$ are pairwise disjoint. So it suffices to show that they are independent in $\mathcal{H}$. Suppose to the contrary that this is not true. By symmetry, we may assume that there exist two distinct vertices $v_{r}, v_{r+1} \in N_{\mathcal{H}}(e\setminus \{u_{r}\}) \cap e$ for some edge $e\in \mathcal{H}$. Note that $\{\{u_1, \ldots, u_{r-1},v_{r}\}, \{u_1, \ldots, u_{r-1},v_{r+1}\},e\}$ is a member in $\Sigma_{r}$, which, by Proposition~\ref{PROP:Tr-free-implies-Sigma-r-free}, is a contradiction. 
\end{proof}

We are now ready to prove Theorem~\ref{THM:main-AES-Tr-codegree}.

\begin{proof}[Proof of Theorem~\ref{THM:main-AES-Tr-codegree}]
     Let $n \ge (r-1)(2r+1)/2$ be an integer. Let $\mathcal{H}$ be an $n$-vertex $\mathrm{T}_{r}$-free $r$-graph satisfying $\delta_{r-1}^+(\mathcal{H})> \frac{2n}{2r+1}$. Note that $\delta_{r-1}^+(\mathcal{H}) \ge r$. Let $V \coloneqq V(\mathcal{H})$. 
     \begin{claim}\label{CLAIM:three-indep-sets}
         There exist $r$ pairwise disjoint independent sets $V_1, \ldots, V_r \subseteq V$ such that 
         \begin{align}\label{equ:min-V1V2V3}
            \min\left\{|V_1|, \ldots, |V_r|\right\} 
            > \frac{2n}{2r+1}. 
        \end{align}
     \end{claim}
     \begin{proof}[Proof of Claim~\ref{CLAIM:three-indep-sets}]
         Fix an edge $e = \{u_1, \ldots, u_{r}\} \in \mathcal{H}$. Let $V_i \coloneqq N_{\mathcal{H}}(e\setminus \{u_i\})$ for $i \in [r]$.
        It follows from Lemma~\ref{LEMMA:F5-3-indep-sets} that $V_1, \ldots, V_r$ are pairwise disjoint and independent in $\mathcal{H}$. Additionally, if follows from $\delta_{r-1}^+(\mathcal{H})> \frac{2n}{2r+1}$ that $\min\left\{|V_1|, \ldots, |V_r|\right\} > \frac{2n}{2r+1}$.
     \end{proof}
     Let $V_1, \ldots, V_r \subseteq V$ be $r$ pairwise disjoint independent sets satisfying~\eqref{equ:min-V1V2V3} and assume that $V_1, \ldots, V_r$ are maximal subject to these constraints. Let $Z \coloneqq V\setminus (V_1 \cup \cdots \cup V_r)$, noting that 
     \begin{align}\label{equ:Z-upper-bound}
         |Z| 
         < n - r\cdot \frac{2n}{2r+1}
         \le \frac{n}{2r+1}.
     \end{align}
     Note that we are done if $Z = \emptyset$. So we may assume that there exists a vertex $v \in Z$. Let $R \coloneqq N_{\mathcal{H}}(v)$. 
     For every $i \in [r]$, let 
     \begin{align*}
         L_{i}
         \coloneqq \left\{e\in L_{\mathcal{H}}(v) \colon \text{$|e\cap V_j| = 1$ for every $j \in [r]\setminus\{i\}$}\right\}
         \quad\text{and}\quad 
         N_i
         \coloneqq R \cap V_i.
     \end{align*}
     It follows from the maximality of $V_1, \ldots, V_r$ that $N_i \neq \emptyset$ for every $i \in [r]$. 
     \begin{claim}\label{CLAIM:R-lower-bound}
         We have $|R| > \frac{2(r-1) n}{2r+1}$.
     \end{claim}
     \begin{proof}[Proof of Claim~\ref{CLAIM:R-lower-bound}]
        Fix an edge $\{v_1, \ldots, v_{r-1}\} \in L_{\mathcal{H}}(v)$.
        For every $i \in [r-1]$, let $e_i \coloneqq \{v_1, \ldots, v_{r-1},v\} \setminus \{v_i\}$ and $U_i \coloneqq N_{\mathcal{H}}(e_i)$. It follows from Property~\eqref{equ:coneighbor-neighbor} that $U_i \subseteq R$ for every $i \in [r-1]$. Combining this with Lemma~\ref{LEMMA:F5-3-indep-sets} and the assumption $\delta_{r-1}^+(\mathcal{H})> \frac{2n}{2r+1}$, we obtain 
        \begin{align*}
            |R|
            \ge |U_1| + \cdots + |U_{r-1}|
            \ge (r-1) \cdot \delta_{r-1}^+(\mathcal{H})
            > \frac{2(r-1) n}{2r+1},  
        \end{align*}
        proving Claim~\ref{CLAIM:R-lower-bound}. 
     \end{proof}
    \begin{claim}\label{CLAIM:Lij-nonempty}
        At most one set in $\{L_{1}, \ldots, L_{r}\}$ is nonempty.
    \end{claim}
    \begin{proof}[Proof of Claim~\ref{CLAIM:Lij-nonempty}]
        Suppose to the contrary that this is not true. By symmetry, we may assume that $L_{1} \neq \emptyset$ and $L_{r} \neq \emptyset$. 
        Fix $\{v_1, \ldots, v_{r-1}\} \in L_{r}$ and $\{w_2, \ldots, w_{r}\} \in L_{1}$, assuming that $(v_1, \ldots, v_{r-1}) \in V_1 \times \cdots \times V_{r-1}$ and $(w_2, \ldots, w_{r}) \in V_2 \times \cdots \times V_r$ (it may hold that $v_i  = w_i$ for $i \in [2,r-1]$). 
        Since $V_1, \ldots, V_r$ are independent in $\mathcal{H}$, we have 
        \begin{align}\label{equ:Nu1u2-Nv2v3}
            N_{\mathcal{H}}(v_1 \cdots v_{r-1}) \subseteq V_r \cup Z
            \quad\text{and}\quad 
            N_{\mathcal{H}}(w_2 \cdots w_r) \subseteq V_1 \cup Z. 
        \end{align}
        Since $\{v_1, \ldots, v_{r-1}\}, \{w_2, \ldots, w_r\} \in L_{\mathcal{H}}(v)$, it follows from~\eqref{equ:Nu-Nvw-empty} that 
        \begin{align*}
            R \cap N_{\mathcal{H}}(v_1 \cdots v_{r-1})
            = R \cap N_{\mathcal{H}}(w_2 \cdots w_r)
            = \emptyset. 
        \end{align*}
        Combining this with~\eqref{equ:Nu1u2-Nv2v3}, Claim~\ref{CLAIM:R-lower-bound}, and the assumption $\delta_{r-1}^+(\mathcal{H})> \frac{2n}{2r+1}$, we obtain 
        \begin{align*}
            &|R \cup  N_{\mathcal{H}}(v_1 \cdots v_{r-1}) \cup N_{\mathcal{H}}(w_1 \cdots w_{r})| \\
            & \quad = |R| + |N_{\mathcal{H}}(v_1 \cdots v_{r-1})| + |N_{\mathcal{H}}(w_1 \cdots w_{r})| - |N_{\mathcal{H}}(v_1 \cdots v_{r-1}) \cap N_{\mathcal{H}}(w_1 \cdots w_{r})|\\
            & \quad > \frac{2(r-1)n}{2r+1} + \frac{2n}{2r+1} + \frac{2n}{2r+1} - |Z|
            > n, 
        \end{align*}
        a contradiction. Here, we used~\eqref{equ:Z-upper-bound} in the last inequality. 
    \end{proof}
    By symmetry, we may assume that $L_{i} = \emptyset$ for every $i \in [2,r]$. Fix a vertex $v_1 \in N_1$ (recall that $N_1 \neq \emptyset$). 
    \begin{claim}\label{CLAIM:find-e}
        There exists an edge $e\in \mathcal{H}$ containing $\{v,v_1\}$ such that 
        \begin{align*}
            |e \cap (V_{2} \cup \cdots \cup V_{r})| \ge  r-2. 
        \end{align*}
    \end{claim}
    \begin{proof}[Proof of Claim~\ref{CLAIM:find-e}]
        Since $v_{1}\in R = N_{\mathcal{H}}(v)$, there exists an edge in $\mathcal{H}$ containing $\{v,v_1\}$. 
        Choose an edge $e\in \mathcal{H}$ containing $\{v,v_1\}$ such that $|e \cap (V_{2} \cup \cdots \cup V_{r})|$ is maximized. Suppose to the contrary that $|e \cap (V_{2} \cup \cdots \cup V_{r})| \le r-3$. 

        Since $V_1, \ldots, V_{r}$ are independent in $\mathcal{H}$, we have $|e\cap V_i| \le 1$ for every $i \in [r]$. In particular, $e\cap V_1 = \{v_1\}$. By symmetry, we may assume that $e\cap V_j \neq \emptyset$ for $j \in [i_{\ast}]$ with some $i_{\ast} \le r-2$. Let $v_j$ denote the vertex in $e\cap V_j$ for $j \in [i_{\ast}]$. 
        Assume that $(e \setminus\{v\}) \cap Z = \{v_{i_{\ast}+1}, \ldots, v_{r}\}$. Let $e' \coloneqq \{v, v_1, \ldots, v_{r-1}\}$. It follows from $\delta_{r-1}^+(\mathcal{H})> \frac{2n}{2r+1}$ and~\eqref{equ:Z-upper-bound} that 
        \begin{align*}
            |N_{\mathcal{H}}(e')|
            > \frac{2n}{2r+1}
            > |Z|,
        \end{align*}
        which implies that there exists $w \in N_{\mathcal{H}}(e') \setminus Z$ such that $\hat{e} \coloneqq e' \cup \{w\} \in \mathcal{H}$. Since $V_1, \ldots, V_{r}$ are independent in $\mathcal{H}$ and $\hat{e} \cap V_j \neq \emptyset$ for every $j \in [i_{\ast}]$, we have $w\in V_{i_{\ast}+1} \cup \cdots \cup V_{r}$. However, since $\hat{e}$ contains $\{v, v_1\}$ and satisfies 
        \begin{align*}
            |\hat{e} \cap (V_{2} \cup \cdots \cup V_{r})|
            = |e\cap (V_{2} \cup \cdots \cup V_{r})| + 1,
        \end{align*}
        this is a contradiction to the maximazlity of $e$. 
    \end{proof}
    It follows from Claim~\ref{CLAIM:find-e} that $L_{j} \neq \emptyset$ for some $j \in [2,r]$, contradicting the assumption that $L_{i} = \emptyset$ for every $i \in [2,r]$. This completes the proof of Theorem~\ref{THM:main-AES-Tr-codegree}. 
\end{proof}
%%%%%%%%%%%%%%%%%%%%%%%%%%%%%
\section{Concluding remark}\label{SEC:remark}
Given integers $r > i \ge 1$, the \textbf{$i$-th shadow} $\partial_{i}\mathcal{H}$ of an $r$-graph $\mathcal{H}$ is defined inductively by letting $\partial_{i}\mathcal{H} \coloneqq \partial\partial_{i-1}\mathcal{H}$ with $\partial_{1}\mathcal{H} \coloneqq \partial\mathcal{H}$.
For every $i$-set $S\subseteq  V(\mathcal{H})$, the \textbf{link} of $S$ in $\mathcal{H}$ is 
\begin{align*}
    L_{\mathcal{H}}(S)
    \coloneqq \left\{e\in \partial_{i}\mathcal{H} \colon S\cup e \in \mathcal{H}\right\}, 
\end{align*}
and the \textbf{degree} of $S$ is $d_{\mathcal{H}}(S) \coloneqq |L_{\mathcal{H}}(S)|$. 
Similar to the definition of $\delta^{+}_{r-1}(\mathcal{H})$, the \textbf{positive minimum $i$-degree of $\mathcal{H}$} is defined as 
\begin{align*}
    \delta^{+}_{i}(\mathcal{H})
    \coloneqq \min\left\{d_{\mathcal{H}}(e) \colon e\in \partial_{r-i}\mathcal{H}\right\}. 
\end{align*}
It is a natural question to explore whether a positive minimum $i$-degree analogue of Theorem~\ref{THM:main-AES-Tr-codegree} holds for $i \in [2,r-2]$ when $r\ge 4$. 
It should be noted that for $r \in \{5,6\}$, classical results by Frankl--F\"{u}redi~\cite{FF89} and an observation by Pikhurko in~\cite{Pik08} show that no such analogue exists for $i=1$.

For integers $\ell \ge r\ge 3$, the \textbf{expansion} $H_{\ell+1}^{r}$ of the complete graph $K_{\ell+1}$ is the $r$-graph obtained from $K_{\ell+1}$ by adding a set of $r-2$ new vertices into each edge, ensuring that different edges receive disjoint $(r-2)$-sets. This is a classical family of hypergraphs introduced by Mubayi in~\cite{Mub06} to extend the seminal Tur\'{a}n Theorem~\cite{TU41} to hypergraphs (see also~\cite{Pik13expansion}). 

A similar argument to that in the proof of Proposition~\ref{PROP:Tr-free-implies-Sigma-r-free} yields the following result. 
\begin{fact}\label{FACT:expansion-codegree}
    Let $\ell \ge r\ge 3$ be integers. 
    Suppose that $\mathcal{H}$ is an $r$-graph satisfying $K_{\ell+1}\not\subseteq \partial_{r-2}\mathcal{H}$ and $\delta_{r-1}^{+}(\mathcal{H}) > (r-2)\binom{\ell}{2}$. Then $\mathcal{H}$ is $H_{\ell+1}^{r}$-free. 
\end{fact}

Using Fact~\ref{FACT:expansion-codegree} and an argument analogous to the third remark after Theorem~\ref{THM:main-AES-Tr-codegree}, one can derive the following tight positive codegree Andr{\'a}sfai--Erd\H{o}s--S\'{o}s theorem for $H_{r+1}^{r}$. 
\begin{theorem}\label{THM:main-AES-expansion-codegree}
    For $n \ge (2r+1) (r-2)\binom{\ell}{2}/2$, every $n$-vertex $H_{r+1}^{r}$-free $r$-graph $\mathcal{H}$ with $\delta_{r-1}^{+}(\mathcal{H}) > \frac{2n}{2r+1}$ is $r$-partite.  
\end{theorem}
%%%%%%%%%%%%%%%%%%%%%%%%%%%%%%
\section*{Acknowledgment}
We would like to thank Yixiao Zhang for his insightful comments, which led to a significant improvement of Theorem~\ref{THM:main-AES-Tr-codegree}. 
%%%%%%%%%%%%%%%%%%%%%%%%%%%%%%%%%%%%%%%%%%%
\bibliographystyle{alpha}%abbrv
\bibliography{CodegreeHypergraphAES}
%%%%%%%%%%%%%%%%%%%%%%%%%%%%%%%%%%%%%%%%%%%%%%%%%
%%%%%%%%%%%%%%%%%%%%%%%%%%%%%
\end{document}